\documentclass{article}
\usepackage[utf8]{inputenc}
\usepackage{amsthm}
\usepackage{amsmath}
\usepackage{amsfonts}
\usepackage{color}
\usepackage{latexsym}
\usepackage{geometry}
\usepackage{enumerate}
\usepackage[shortlabels]{enumitem}
\usepackage{csquotes}
\usepackage{graphicx}
\usepackage{comment}
\usepackage{appendix}
\usepackage{hyperref}
\hypersetup{
    colorlinks=true,
    linkcolor=blue,
    filecolor=magenta,      
    urlcolor=cyan,
}
\usepackage{bm}
\usepackage{amssymb}
\usepackage{tikz-cd}

\emergencystretch=\maxdimen
\def \NN {\mathcal{N}}

\def \XX {\mathcal{X}}

\def \dist {{\rm dist}}
\def \IR {\mathbb{R}}
\def \IS {\mathbb{S}}
\def \IN {\mathbb{N}}

\def \Bry {M_{\rm Bry}}
\def \Per {M_{\rm Per}}

\def \eps {\epsilon}

\def \IN {\mathbb{N}}

\DeclareMathOperator{\Rm}{Rm}
\DeclareMathOperator{\Ric}{Ric}

\makeatletter
\newcommand*{\rom}[1]{\rm {\expandafter\@slowromancap\romannumeral #1@}}
\makeatother

\def\XXint#1#2#3{{\setbox0=\hbox{$#1{#2#3}{\int}$ }
\vcenter{\hbox{$#2#3$ }}\kern-.6\wd0}}

\protected\def\vts{%
  \ifmmode
    \mskip0.5\thinmuskip
  \else
    \ifhmode
      \kern0.08334em
    \fi
  \fi
}

\numberwithin{equation}{section}
\newtheorem{Theorem}{Theorem}[section]
\newtheorem{Proposition}[Theorem]{Proposition}
\newtheorem{Conjecture}[Theorem]{Conjecture}
\newtheorem{Lemma}[Theorem]{Lemma}
\newtheorem{Corollary}[Theorem]{Corollary}
\newtheorem{Remark}[Theorem]{Remark}
\theoremstyle{definition}
\newtheorem{Definition}[Theorem]{Definition}

\def \NN {\mathcal{N}}

\def \IS {\mathbb{S}}

\title{Dimension Reduction for Positively Curved Steady Solitons}
\author{Pak-Yeung Chan\footnote{Pak-Yeung Chan's research is supported by EPSRC grant EP/T019824/1.}, Zilu Ma, Yongjia Zhang\footnote{Yongjia Zhang's research is partially supported by Shanghai Sailing Program 23YF1420400 and Research Start-up Fund of SJTU WH220407110.}}
\date{}

\begin{document}

% \thanks{* University of California San Diego}

\maketitle
\begin{abstract}
We consider noncollapsed steady gradient Ricci solitons with nonnegative sectional curvature. We show that such solitons always dimension reduce at infinity. 
% Namely, along any sequence going off to spatial infinity, the canonical Ricci flow, after scaling with the scalar curvature and possibly passing to a subsequence, converges to a limit flow that splits off a line. 
This generalizes an earlier result in \cite{CDM22} to higher dimensions. %and drops the condition of curvature decay but with a stronger curvature condition. 
In dimension four, we classify possible reductions at infinity, which lays foundation for possible classifications of steady solitons. Moreover, we show that any tangent flow at infinity of a general noncollapsed steady soliton must split off a line. This generalizes an earlier result in \cite{BCDMZ21} to higher dimensions. While this article is under preparation, we realized that part of our main results are proved independently in a recent post \cite{ZZ23} 
% \textcolor{red}
{under different assumptions}.
\end{abstract}

\bigskip % \bigskip

\tableofcontents
\section{Introduction}

\def \rd {\mathrm{d}}

In order to apply Ricci flow to topological problems especially in dimension 4, one needs better understanding of the formation of singularities in order to perform surgeries. Among singularity models, \textit{steady Ricci solitons} play a crucial role as they may arise from type-II singularities \cite{Ha95}. Moreover, the classification for steady solitons is fundamental to the classification of general ancient flows, see \cite{Bre13,Bre20,ABDS22,BDS21}.

We call a triple $(M^n,g,f)$ a complete \textbf{steady gradient Ricci soliton}, or \textbf{steady soliton},
if $(M^n,g)$ is a complete Riemannian manifold and $f$ is a smooth function on $M$ satisfying
\[
    \Ric=\nabla^2f.
\]
A Ricci flat manifold is an obvious trivial example of steady soliton, but we shall always consider the non-Ricci-flat case. By a classical result due to Hamilton \cite{Ha95}, we may normalize the metric so that
\[
    R+|\nabla f|^2=1.
\]
Let $(\Phi_t)_{t\in \IR}$ be the $1$-parameter group of diffeomorphisms generated by $-\nabla f,$ then $g_t=\Phi_t^*g$ is a Ricci flow, called the \emph{canonical Ricci flow} induced by the steady soliton $(M^n,g,f).$ The time span of $(M,g_t)$ is $\IR$, and any such solution is called \textit{eternal} by Hamilton. In fact, each time-slice of the canonical flow is also a steady soliton: let $f_t:=f\circ\Phi_t$, then we obvisouly have
\begin{align}\label{steady-soliton-canonical-flow-identity}
    \Ric_{g_t}=\nabla^2_{g_t}f_t,\quad R_{g_t}+|\nabla_{g_t}f_t|^2=1,\quad \text{ for all }\quad t\in(-\infty,+\infty).
\end{align}

We say that a steady soliton $(M^n,g,f)$ is \textbf{noncollapsed}, if the canonical Ricci flow induced by $(M^n,g,f)$ has bounded Nash entropy, i.e.,
\[
    \mu_\infty := \inf_{\tau>0} \NN_{o,0}(\tau) >-\infty,
\]
where $\NN_{o,0}(\tau)$ denotes the pointed Nash entropy based at a space-time point $(o,0)$ at scale $\tau$; see, e.g., \cite{Bam20a}. Nash entropy is the fundamental monotone quantity in Bamler's works \cite{Bam20a,Bam20c}. By a result of the second and the third authors \cite[Theorem 1.13]{MZ21}, our notion of noncollapsedness is equivalent to the $\kappa$-noncollapsedness \cite{Per02} for the canonical Ricci flow of a steady soliton with nonnegative sectional curvature. 

% We often study steady solitons satisfying the following condition.
% \begin{equation}
% \tag{Curv}
% \label{cond: Ric and pinch}
%     \Ric\ge 0,\quad
%     |{\Rm}|\le \Lambda R,
% \end{equation} 
% everywhere on $M,$ for some constant $\Lambda<\infty$. Clearly, this condition holds if $\sec\ge 0$.

Choi, Haslhofer, and Hershkovits in \cite{CHH23} fully classified $\alpha$-noncollapsed translators in $\IR^4$, which are natural analogues of steady solitons in the extrinsic setting.
Among other conjectures concerning general noncollapsed ancient flows,  Haslhofer \cite{Has23} proposed the following conjecture about $4$-dimensional noncollapsed steady solitons.

\begin{Conjecture}
    
Any $4$-dimensional noncollapsed steady soliton with $\sec\ge 0$ can only be one the following: the 4d Bryant soliton, the 3d Bryant soliton times a line, or one of the 1-parameter family of $\mathbb{Z}_2\times {\rm O}_3$-symmetric steady
solitons constructed by Lai \cite{Lai20}.

\end{Conjecture}

We remark that this is slightly stronger than Haslhofer's original conjecture, which assumes $\Rm\ge 0$ (nonnegative curvature operator). One of the reasons why this stronger version is also likely is because the compactness result given in the appendix of the current paper assumes only $\sec\ge 0$ (nonnegative sectional curvture). 

According to Theorem \ref{thm: tan flow oo} in \cite{BCDMZ21}, any noncollapsed non-Ricci-flat 4d steady solitons with bounded curvature admits a unique tangent flow at infinity, which falls into one of the following: $(\IS^3/\Gamma)\times \IR, \IS^2\times \IR^2, (\IS^2\times_{\mathbb{Z}_2}\IR)\times \IR$. 
Assuming $\sec\ge 0$ and $\Ric>0,$  Corollary \ref{cor: 4d tan flow} below shows that the tangent flow at infinity can only be $\IS^3\times \IR$ or $\IS^2\times \IR^2$;  
in the former case, the steady soliton must be the Bryant soliton by Theorem \ref{thm: main cyl flow}, which is in fact implied by \cite[Theorem 1.3]{BCDMZ21} and \cite[Theorem 1.2]{Bre14}. So, in the majority of the paper, we  consider steady solitons with $\IS^2\times \IR^2$ as the unique tangent flow at infinity, which conjecturally should be Lai's examples in \cite{Lai22}.

As a preliminary step towards the classification of steady solitons, we shall understand geometric limits at infinity. For example, in the recent classification results for 3d ancient Ricci flows \cite{Bre13,Bre20,ABDS22,BDS21, Lai22},  the geometric limits were classified at first. 
We first prove dimension reduction for steady solitons in any dimensions, assuming $\sec\ge 0$.
\begin{Theorem}
\label{thm: main}
    Let $(M^n,g,f)$ be a complete noncollapsed steady gradient Ricci soliton with $\sec\ge 0$. If $(M^n,g)$ is non-flat, then it dimension reduces at infinity. That is, for any points $x_i\to \infty,$ let $r_i^{-2}=R(x_i)$. 
    By passing to a subsequence,
    \[
        (M, r_i^{-2}g_{r_i^2t}, (x_i,0))
        \longrightarrow (N^{n-1}\times \IR, \bar g_t+\rd z^2, ((\bar x,0),0)),
    \]
    in the sense of Cheeger-Gromov-Hamilton over $(-\infty,0]$,
    where $(N,\bar g_t)$ is some noncollapsed ancient flow, and $g_t$ denotes the canonical Ricci flow induced by the steady soliton.
\end{Theorem}

In fact, we shall prove a pre-compactness result, namely, Theorem \ref{thm: cpt}, for a larger class of noncollapsed ancient flows. As a consequence, we can prove a slightly more general splitting theorem at infinity---Theorem \ref{thm: general split}. Theorem \ref{thm: main} follows as a special case. We follow Perelman's arguments in \cite[Section 11]{Per02} and Brendle's recent presentation in \cite{Bre22}.

In dimension 4, thanks to the recent full classification of 3d noncollapsed ancient flows (\cite{Bre20,BDS21}),
we can classify possible reductions.
When assuming $\sec\ge 0,\Ric>0,$ by Corollary \ref{cor: 4d tan flow} below, the remaining interesting case is the one in which steady solitons has $\IS^2\times \IR^2$ as the tangent flow at infinity.
\begin{Theorem}\label{thm: main 2}
     Let $(M^4,g,f)$ be a complete noncollapsed steady gradient Ricci soliton with $\sec\ge 0, \Ric>0$. If the tangent flow at infinity is $\IS^2\times \IR^2$, then it can only dimension reduce to
     \[
        \IS^2\times \IR \quad
        \text{ or }
        \quad
        M^3_{\rm Bry}.
     \]
\end{Theorem}

By a slight abuse of notations, we denote by
\[
    \IS^2\times \IR,\quad
     M^{3}_{\rm Bry},\quad
     M^{3}_{\rm Per}
\]
the standard shrinking $\IS^2\times \IR$, the canonical form of the three-dimensional Bryant soliton, Perelman's type-II compact $\kappa$-solution \cite[1.4]{Per03a}, respectively. For simplicity, we sometimes omit the metrics when they are clear from the context.

\medskip
\noindent
\textbf{Remark.} \emph{We point out that some results similar to Theorem \ref{thm: main} and Theorem \ref{thm: main 2} are proved independently in a recent post \cite{ZZ23} under different assumptions.}
\medskip

In a forthcoming work \cite{MMS23}, Mahmoudian, \v{S}e\v{s}um, and the second-named author will obtain precise geometric asymptotics for 4d steady solitons with ${\rm O}_3$-symmetry, which are similar to \cite{ABDS22}. Theorem \ref{thm: main 2} is fundamental to such precise analysis.

%In general, we can show that any tangent flow at infinity of the canonical flow generated by a steady soliton splits off a line. This result does not have any curvature restriction.

The idea of proving Theorem \ref{thm: main} and Theorem \ref{thm: main 2} is to show that the limit admits a line, by making use of the curvature positivity condition. The advantage of this approach is that one can obtain a splitting limit from any sequence of points going to infinity, and the disadvantage is the restriction of the curvature condition. However, there is another approach presented in \cite{BCDMZ21}, which is to show that the potential function converges to a splitting map. This method works without any curvature restriction, but holds only for tangent flows at infinity. We also apply this method to higher dimensions and prove the following general result.

\begin{Theorem}[{= Theorem \ref{thm: split tan flow}}]
    Let $(M^n,g,f)$ be a complete non-Ricci-flat noncollapsed steady soliton with bounded curvature.
    Then any tangent flow at infinity of its canonical flow splits off a line.
\end{Theorem}

Noted that the tangent flow at infinity \cite{Bam20c} is an $\mathbb{F}$-limit of a scaled sequence of Ricci flows; it is a metric soliton which possibly has a singular set with codimension at least four. Tangent flow at infinity is not unique a priori; it depends on the scaling sequence but is independent of the base point \cite{CMZ21a}. Therefore, if one tangent flow at infinity is smooth, there could possibly be other non-smooth tangent flows. However, we shall prove the uniqueness of tangent flow for the canonical flow of a steady soliton, given that one of its tangent flows is a cylinder. This generalizes a result in \cite{BCDMZ21}. We remark that the following theorem is not a direct consequence of \cite{CM22}, since tangent flows are metric solitons instead of smooth solitons.

\begin{Theorem}[{= Theorem \ref{thm: cyl flow}}]
\label{thm: main cyl flow}
    Let $(M^n,g,f)$  be a complete non-Ricci-flat noncollapsed steady soliton with bounded curvature.
    If one tangent flow at infinity is $(\IS^{n-1}/\Gamma)\times \IR$, then it is the unique one. Furthermore, for any $\eps>0$, there is a compact set $K_\eps$ such that every point outside $K_\eps$ is a center of an $\eps$-neck. Moreover, outside $K_\eps$ we have,
    \[
        \Rm>0,\qquad
        r R = \tfrac{n-1}{2} + o(1),
        % \frac{C_n-\eps}{r}
        % \le R \le \frac{C_n+\eps}{r},
    \]
    as $x\to \infty,$
    where $r(x)=|ox|$ for some  point $o\in M.$
    If, in addition, $\sec\ge 0$ and $\Ric>0,$ then $\Gamma=1$ and $(M^n,g,f)$ is isometric to the Bryant soliton up to rescaling.

\end{Theorem}

\section{Preliminaries}

% In the following, we denote by
% \[
%     B_t(x,r)
% \]
% the ball at time $t$ with respect to $g_t$, the canonical Ricci flow induced by the steady soliton $(M,g,f).$ We denote by
% \[
%     P^-(x,t,r):=B_t(x,r)\times[t-r^2,t]
% \]
% the standard backward parabolic ball.

\def \CC {\mathcal{C}}
We denote by
\[
    \CC_f:=\{x: \nabla f(x)=0\}
\]
the set of all critical points of $f.$

% \begin{Lemma}
%     Let $(M^n,g,f)$ be a steady soliton with $\Ric\ge 0.$ If $\CC_f\neq \emptyset,$ then $\CC_f$ is compact.
%     Suppose further that $\Ric>0,$ then $\CC_f$ consists of at most one point.
% \end{Lemma}
% \begin{proof}

%     If $\Ric>0,$ then $\nabla^2f=\Ric>0$, and thus $f$ is a strictly convex function. So it has a unique critical point if $\CC_f\neq \emptyset.$
% \end{proof}

\def \bn {\mathbf{n}}
\begin{Lemma}
\label{lem: sec level set}
    Let $(M^n,g,f)$ be a steady soliton with $\Ric>0.$ 
    Then $|\CC_f|\le 1$, i.e. $f$ has at most $1$ critical point.
    There is a compact subset $K\subset M$, such that
    if  $M_s:=\{f=s\}$ does not intersect $K$, then $M_s$ is a smooth manifold and
    \[
        \sec_{M_s}>\sec_M.
    \]
    As a consequence, if $n=4,$ then $M_s$ is diffeomorphic to either $\IS^3$ or $\IR^3.$
\end{Lemma}
\begin{proof}
$f$ is strictly convex, because $\nabla^2f=\Ric>0.$ 
Thus, if $\CC_f\neq \emptyset$, then $|\CC_f|=1$.
In this case, by \cite[Lemma 2.1]{DZ21}, each level set is diffeomorphic to $\IS^{n-1}.$

Now for both cases, 
$|\CC_f|=0$ or $1,$ we may assume that $|\nabla f|>0,$ outside $B(o,1),$ for some $o\in M.$ Since $\nabla^2f=\Ric>0,$ each level set $M_s=\{f=s\}$ is a smooth manifold.
    % Recall that if $\CC_f\neq\emptyset$, it must be compact, and we may assume that $\CC_f\subseteq B(o,A)$, for some $o\in M, A>0.$ Thus, $|\nabla f|>0$ outside $B(o,A).$ $M_s$ is smooth if $M_s\cap B(o,A)=\emptyset.$

    We write $\bn:=\frac{\nabla f}{|\nabla f|}.$ The second fundamental form ${\rm II}$ of $M_s$ is given by
    \[
        {\rm II}(X,Y)=
        \langle \nabla_X \bn, Y \rangle=\frac{1}{|\nabla f|}\Ric(X,Y),
    \]
    for any $X,Y$ tangent to $M_s.$
    Thus, ${\rm II}>0$. By the Gauss formula,
    \begin{align*}
        \sec_{M_s}(X,Y)
        &= R(X,Y,Y,X) + {\rm II}(X,X){\rm II}(Y,Y)
        - \left({\rm II}(X,Y)\right)^2
        > \sec_M(X,Y),
    \end{align*}
    for any unit vectors $X,Y$ tangent to $M_s$ with $X\perp Y.$ Here, we used the Cauchy-Schwarz inequality, since ${\rm II}>0,$ and $X\perp Y.$

    In particular, if $n=4, M_s$ is a $3$-manifold with $\Ric_{M_s}>0.$ Thus, $M_s$ is diffeomorphic to $\IS^3$ or $\IR^3$ by classical results.
\end{proof}

In dimension $4$, let us recall the following theorem in \cite{BCDMZ21}. Although the original statement was formulated for steady solitons that arise as singularity models, yet in the justifications made in \cite[Appendix A]{Bam21},  Bamler indicates that all the theories developed in \cite{Bam20a} and \cite{Bam20c} can be generalized to complete Ricci flows with bounded curvatures on compact intervals.

\begin{Theorem}[{\cite{BCDMZ21}}]
\label{thm: tan flow oo}
    Let $(M^4,g,f)$ be a complete nontrivial and noncollapsed steady gradient Ricci soliton with bounded curvature. Then its tangent flow at inifinity is unique and falls into one of the following:
    \[
        (\IS^3/\Gamma)\times \IR,\quad
        \IS^2\times\IR^2,\quad
        (\IS^2\times_{\mathbb{Z}_2} \IR)\times \IR.
    \]
\end{Theorem}

\section{Growth of the Potential Function}

\def \CC {\mathcal{C}_f}
\def \Rc {{\rm Ric}}

This section is taken partly from the second named author's dissertation \cite[\S 7.2]{M22}.

For a steady gradient Ricci soliton $(M^n,g,f)$, it is obvious that if $R$ has uniform decay, then the critical set of $f$
\[
	\CC := \{x\,|\,\nabla f(x)=0\}
\]
is nonempty and compact. The converse is incorrect. For example, consider the product of two Bryant solitons.
Thus, it is strictly weaker than the uniform decay of $R$ to assume that $\CC$ is nonempty and compact.

% One of the goals is to show that for a steady GRS with, e.g., positive sectional curvature, $f$ has a linear growth. It seems that there are no counterexamples for now.

% \begin{Lemma}
% Suppose that $(M^n,g_t)_{t\in [0,1]}$ is a complete Ricci flow.
% \end{Lemma}

The following arguments are indeed contained in \cite[Section 9]{CDM22}. We shall record them for future use.
\begin{Proposition}
\label{prop: linear growth Ric}
Suppose that $\Rc\ge 0$ everywhere on $M$.
% either 
% \begin{enumerate}
% 	\item $\Rc\ge 0$ everywhere on $M$; or
% 	\item $\Rc(\nabla f,\nabla f)\ge 0$ outside a compact set and $f\ge -C$ for some $C\ge 0.$
% \end{enumerate}
If $\CC$ is nonempty and compact, then $f$ has linear growth.
\end{Proposition}

\begin{proof}
% Assume that $K\subset B_A(o)$ 
% for some $A<\infty.$
% For any $x\notin \CC,$ we first claim that 
% $\{\Phi_t(x):t\ge 0\}$ is bounded.

% Let $r_0 = d(x,o)+2A.$ 
% Suppose that there is a first time $t_0>0$ such that $\Phi_{t_0}(x)\notin B_{r_0}(o).$ 
% Then for any $t>t_0,$
% \[
%     d(\Phi_t(x),o)=
%     d_t(x,o) \le d(x,o).
% \]
Pick $o\in \CC.$ 
Suppose that $\CC\subset B(o,A)$ for some $A>0.$
\\

% We may enlarge $A$ so that in case 2, $\Ric(\nabla f, \nabla f)\ge 0$ outside $B(o,A).$ 
% We claim that $\Phi_t(x)\to \CC$ as $t\to \infty.$ 

\noindent
\textbf{Claim:} For any $r\ge 2A$ and any $x\notin \bar B(o,r),$ there is $t_x>0,$ such that 
\[
	\Phi_{t_x}(x)\in \partial B(o,r),;
	\qquad \Phi_{t}(x)\notin \bar B(o,r),,\quad t\in [0,t_x).
\]
\begin{proof}[Proof of the Claim]
Since $\Rc\ge 0$ on $M$, for $t\ge 0$ and any $x\in M,$
\[
	\dist(o,\Phi_t(x))=|ox|_t \le |ox|,
\]
by the standard distance distorsion estimate. (See, e.g., \cite[Lemma 17.13]{Ha95}.)
So the curve $\Phi_t(x)$ stays bounded for $t\ge 0.$

We show that $\Phi_t(x)$ accumulates on $\CC.$  Suppose to the contrary that there is a subsequence $t_j\to \infty, \Phi_{t_j}(x)\to y$ for some $y$ but $y\notin \CC.$ 
Let $2c=|\nabla f|^2(y)>0.$ For $t\ge 0,$
\[
	\partial_t |\nabla f|^2(\Phi_t(x)) = -2 \Rc(\nabla f,\nabla f)(\Phi_t(x))\le 0.
\]
Thus for large $j,$
\[
	f(\Phi_{t_j}(x))-f(x) = -\int_{0}^{t_j} |\nabla f|^2(\Phi_s(x))\, ds 
	\le -|\nabla f|^2(\Phi_{t_j}(x))t_j\le -c t_j,
\]
which is impossible since $f(\Phi_{t_j}(x))\to f(y)$ while $t_j\to \infty$. 
Hence $\Phi_t(x)$ accumulates on $\CC.$ Since $\CC\subset B(o,A),$ there must be such a number $t_x>0$ for each $x\notin \bar B(o,r),.$

% \noindent
% \textbf{Case 2}: Suppose to the contrary that $\Phi_t(x)\notin\bar B(o,r)$ for all $t\ge 0.$
% Similar to the argument above, for any $t\ge 0,$ by the assumtion,
% \begin{align*}
% 	-C-f(x)&\le f(\Phi_{t}(x))-f(x) = -\int_{0}^{t} |\nabla f|^2(\Phi_s(x))\, ds.
% \end{align*}
% It follows that
% \[
% 	\int_{0}^{\infty} |\nabla f|^2(\Phi_t(x))\, dt \le f(x)+C<\infty.
% \]
% Due to the monotonicity of $|\nabla f|^2(\Phi_t(x)),$
% \[
% 	\lim_{t\to \infty} |\nabla f|^2(\Phi_t(x)) = 0.
% \]
% which is impossible since $t$ can be arbitrarily large.
% Hence such a number $t_{x,r}$ exists.

\end{proof}

% Suppose that $\CC\subseteq B(o,r_0-1),$ for some $r_0>0.$ For any $x\notin B(o,r_0),$ there is a first entering time $t_x>0$ so that
% \[
% 	\Phi_{t_x}(x)\in \partial B(o,r_0);
% 	\qquad \Phi_{t}(x)\notin B(o,r_0),\quad t\in (0,t_x).
% \]
Now we fix $r_0\ge 2A.$
Let $C, \theta>0$ be constants such that
\[
	\min_{ \partial B(o,r_0)} f\ge -C,\quad
	\min_{\partial B(o,r_0)} |\nabla f|\ge \theta.
\]
Then
\[
	\begin{split}
	f(x) &= f(\Phi_{t_x}(x)) + \int_0^{t_x} |\nabla f|^2(\Phi_s(x))ds\\
	&\ge -C + |\nabla f|(\Phi_{t_x}(x)) \int_0^{t_x} |\partial_s\Phi_s(x)| ds \\
	&\ge -C + \theta \cdot d(x, \partial B(o,r_0))
	\ge \theta |ox| - (C+\theta r_0).
	\end{split}
\]
Thus, $f$ has linear growth.
\end{proof}

Recall that a function is said to be proper, if $\{f\le c\}$ is compact, for any $c\in \IR.$

\begin{Lemma}
Suppose that $\Rc\ge 0$ outside a compact set 
and $f$ is proper. Then $\CC$ is nonempty and compact.
\end{Lemma}
\begin{proof}

Suppose that $\Ric\ge 0$ outside $B(o,A)$ for some $A>0,o\in M.$

As $f$ is proper, it has a minimum and thus $\CC$ is nonempty. Assume by contradiction that $\CC$ is not compact. Then there is a sequence $x_j\in \CC$ and $x_j\to \infty.$
By \cite[Lemma 9.7]{CDM22}, we may assume that there is $x_0\in \CC$ such that, for each $j\ge 1$, $x_j\in \CC$ and any minimal geodesic from $x_0$ to $x_j$ does not intersect $B(o,A).$ 
Let $\gamma_j:[0,\ell_j]\to M$ be a normal minimizing geodesic joining $x_0$ to $x_j$.
Note that $\Phi_t(x_j)\equiv x_j$ for all $t$ since $x_j\in \CC.$ Thus, for any $t\in\IR,$
\[
	|x_0x_j|_t = \dist(\Phi_t(x_0),\Phi_t(x_j)) = |x_0x_j|=\ell_j.
\]
We denote by $L_t(\gamma)$ the length of a curve $\gamma$ with respect to metric $g_t.$ 
For $t\in (-1,1),$
\[
	L_t(\gamma_j)\ge |x_0x_j|_t = \ell_j,\quad
	L_0(\gamma_j)=\ell_j. 
\]
So
\[
	\left.\frac{d}{dt}\right|_{t=0} L_t(\gamma_j) = 0.
\]
On the other hand, by the standard distance distorsion formula under Ricci flow \cite[Section 17]{Ha95},
\[
	0=\left.\frac{d}{dt}\right|_{t=0} L_t(\gamma_j)
	= - \int_{0}^{\ell_j} \Ric(\dot\gamma_j,\dot\gamma_j)\Big|_{\gamma_j(s)}\, ds.
\]
Since $\gamma_j$ does not intersect $B(o,A)$, we have $\Ric(\dot\gamma_j,\dot\gamma_j)|_{\gamma_j(s)}\ge 0$ for $s\in [0,\ell_j].$ 
It follows that
\[
	\Ric(\dot\gamma_j,\dot\gamma_j)|_{\gamma_j(s)} = 0,
\]
for all $s\in [0,\ell_j].$
% Fix $o\in \CC$ and let $\gamma_j:[0,\ell_j]\to M$ be a normal minimizing geodesic joining $o$ to $x_j$. 
% Note that $\Phi_t(x_j)=0$ for all $t$ as $x_j\in \CC.$ Hence
% \[
% 	d_{g(t)}(x_j,o)=d(\Phi_t(x_j),\Phi_t(o)) = d(x_j,o)=l_j.
% \]
% We denote by $L_t(\gamma)$ the length of a curve $\gamma$ with respect to $g(t).$
% Since
% \[
% 	\partial_t L_t(\gamma_j) = - \int_0^{l_j} \frac{\Rc(\dot\gamma_j,\dot\gamma_j)}{|\dot\gamma_j|_{g(t)}}(\gamma_j(s))
% 	ds \le 0, 
% \]
% for $t\ge 0,$
% \[
% 	L_t(\gamma_j) \le l_j.
% \]
% But
% \[
% 	L_t(\gamma_j) \ge d_{g(t)}(x_j,o) = d(x_j,o)=l_j.
% \]
% So $\gamma_j$'s are also minimizing with respect to $g(t)$ for $t\ge 0.$ It follows that
% \[
% 	\Rc(\dot \gamma_j, \dot \gamma_j)(\gamma_j(s)) = 0,\quad\forall s\in [0,l_j].
% \]
So $\dot \gamma_j(s)$ is a null eigenvector of $\Rc|_{\gamma_j(s)}$ and 
\[
	\partial_s R(\gamma_j(s)) = - 2\Rc(\dot \gamma_j,\nabla f)(\gamma_j(s))=0,\quad
	\forall s\in [0,\ell_j].
\]
Then
\[
	|\nabla f|(\gamma_j(s)) \equiv |\nabla f|(x_0) = 0,\quad \forall s\in [0,\ell_j].
\]
Thus, for any $j\ge 1,$
\[
	f(x_j) = f(x_0),
\]
which is a contradiction to the assumption that $f$ is proper.

\end{proof}

In summary, we have the following.
\begin{Proposition}
\label{prop: linear growth cond}
Suppose $\Rc\ge 0$ on a complete steady gradient Ricci soliton $(M^n,g,f)$. Then the following are equivalent:
\begin{enumerate}[(a)]
	\item $f$ has linear growth;
	\item $\CC$ is nonempty and compact;
    \item $f$ is proper;
    % \item There is a sequence $x_i\to \infty$, such that each $\{f=f(x_i)\}$ is compact.
\end{enumerate}
\end{Proposition}

\section{Geometry at Infinity}
In dimension $4$, we have finer descriptions of the geometry at infinity thanks to the recent full classification of $\kappa$-solutions in dimension $3.$

Our goal in this section is to prove the following.
\begin{Theorem}
\label{thm: possible limits 4D}
    Let $(M^4,g,f)$ be a noncollapsed steady soliton with $\sec\ge 0,\Ric>0.$ If its tangent flow at infinity is $\IS^2\times \IR^2,$ then it can only dimension reduces to
    \[
        \IS^2\times \IR\quad\text{ or }\quad
        M^{3}_{\rm Bry}.
    \]
\end{Theorem}

We summarize the basic idea of the proof of Theorem \ref{thm: possible limits 4D}.
Let $(M^4,g,f)$ be a steady soliton as in Theorem \ref{thm: possible limits 4D}.
By Theorem \ref{thm: main}, it always dimension reduces. 
By the full classification of noncollapsed ancient flows, \cite{Bre20,BDS21}, it can only dimension reduce to the following:
\[
    \IS^3/\Gamma, \quad \IS^2\times \IR,\quad
\IS^2\times_{\mathbb{Z}_2} \IR,\quad
\Bry^3, \quad \Per^3.
\]
Our goal is to rule out $\IS^3/\Gamma$, 
 $\IS^2\times_{\mathbb{Z}_2} \IR$, and $\Per^3$ as possible limits. First of all, $\IS^2\times_{\mathbb{Z}_2} \IR$ cannot be a possible limit, since otherwise $\mathbb{RP}^2$ can be embedded into some level set $\{f=s\}$, which is diffeomorphic to either $\IS^3$ or $\IR^3$; this is impossible. Secondly, $\IS^3/\Gamma$ or $\Per^3$ cannot be a possible limit, since otherwise it prevents some level set $\{f=s\}$ from having arbitrarily long necks; this argument roughly follows from \cite[Section 2]{BDS21}.

For clarity, we divide the proof of Theorem \ref{thm: possible limits 4D} into several subsections.

\medskip
\noindent\textbf{Background assumptions:}
In the rest of this section, unless otherwise stated, $(M^4,g,f)$ is a noncollapsed steady soliton with 
\[
   \sec\ge 0,\quad \Ric>0. 
\]
Suppose also that its tangent flow at infinity is $$\IS^2\times \IR^2.$$

\subsection{Ruling out \texorpdfstring{$\IS^2\times_{\mathbb{Z}_2}\IR$}{Z2}}

We first rule out $\IS^2\times_{\mathbb{Z}_2}\IR$ as a possible limit. 
\begin{Proposition}
  \label{prop: S^2 x R / Z_2}
    $(M^4,g,f)$ cannot dimension reduce to $\IS^2\times_{\mathbb{Z}_2}\IR$.  
\end{Proposition}
The following lemma is needed in the proof of the above proposition.
\begin{Lemma}
\label{lem: conv to a shrinker}
    Let $(M^n,g,f)$ be a noncollapsed steady soliton. Suppose that for some $x_i\to \infty,$
    if $r_i^{-2}:=R(x_i),$
    \[
        (M,r_i^{-2}g_{r_i^2t}, x_i)
        \longrightarrow (\overline{M},\bar g_t, \bar x),
    \]
    locally smoothly over $(-\infty,0],$ where  $(\overline{M},\bar g_t)$ is a Ricci flow admitting a shrinking gradient Ricci soliton struture. Then
    \[
        R(x_i)\to 0.
    \]
    Moreover, the limit flow splits off a line:
    \[
     (\overline{M},\bar g_t)= (\IR\times N, dz^2+h_t),
    \]
    where $z$ denotes the standard coordinate on $\IR,$
    and 
    \[
        f_i := r_i^{-1}(f-f(x_i))
    \]
    converges locally smoothly to the coordinate function $z$.
    
\end{Lemma}

Assuming this Lemma, we can rule out $\IS^2\times_{\mathbb{Z}_2}\IR$.
\begin{proof}[Proof of Proposition \ref{prop: S^2 x R / Z_2}]
   Suppose that for some $x_i\to \infty,$
    if $r_i^{-2}=R(x_i),$
    \[
        (M,r_i^{-2}g_{r_i^2t}, x_i)
        \to ((\IS^2\times_{\mathbb{Z}_2}\IR)\times \IR, (0,\bar x)),
    \]
    locally smoothly over $(-\infty,0].$ By Lemma \ref{lem: conv to a shrinker}, $f_i=(f-f(x_i))/r_i$ converges to the coordinate function $z$ of the last $\IR$-factor. Therefore, the level set $\{f=f(x_i)\}$, after scaling, converges to $\IS^2\times_{\mathbb{Z}_2}\IR$, locally smoothly. In particular, $\mathbb{RP}^2$, as the tip of $\IS^2\times_{\mathbb{Z}_2}\IR$, can be embedded into $M_i:=\{f=f(x_i)\}$, for large $i.$
   % If $\CC_f\neq \emptyset,$ then $M_i$ is diffeomorphic to $\IS^3.$ If $\CC_f= \emptyset,$ then $M_i$ is noncompact with $\sec_{M_i}>\sec_M\ge 0,$ by Lemma \ref{lem: sec level set}. Thus, $M_i$ is diffeomorphic to $\IR^3$ by the classical theorem of Gromoll and Myer.
   However, $\mathbb{RP}^2$ cannot be embedded into $\IS^3$ or $\IR^3.$ This is  a contradiction.

\end{proof}

Now we prove Lemma \ref{lem: conv to a shrinker}.
\begin{proof}
The proof is similar to \cite[Proposition 3.1]{BCDMZ21}.

    Suppose to the contrary that $R(x_i)\ge a$, for some constant $a>0$, by possibly passing to a subsequence. 
    By the assumption, there is a smooth function $\phi$ such that
    \[
        \Ric_{\bar g_0} +\nabla^2_{\bar g_0} \phi = \tfrac{1}{2}\bar g_0.
    \]
    By \cite[Lemma 9.1]{CDM22}, since $r_i^{-2}\ge a,$ 
    $(\overline{M},\bar g_0)$ also admits a steady gradient Ricci soliton structure with potential function $\bar f,$ i.e., 
    $$\Ric_{\bar g_0}=\nabla^2_{g_0}\bar f.$$
    Define $\rho = \phi+\bar f.$ Then
    \[
        \nabla^2_{g_0} \rho = \tfrac{1}{2}g_0,
    \]
    and thus $(\overline{M},\bar g_0)$ must be the standard $\IR^n.$ This is a contradiction to the fact that $R_{\bar g_0}(\bar x)=1.$ Thus, $r_i^{-2}=R(x_i)\to 0.$

    Finally, to prove the splitting result of the limit flow, we consider only the finial time-slices of the scaled flows $(M^n,r_i^{-2}g)$. We write $g_i:= r_i^{-2}g, f_i=(f-f(x_i))/r_i.$ Then
    \[
        f_i(x_i)=0,\quad
        |df_i|^2_{g_i}(x_i)=1-R(x_i)\to 1.
    \]
    Moreover,
    \[
        |\nabla^2_{g_i} f_i|_{g_i} = r_i^{-1} |{\Ric}_{g_i}|_{g_i}
        \to 0,
    \]
    uniformly over compact subsets, due to the smooth convergence. Similarly, we obtain the convergence of higher order derivatives of $f_i$. Therefore, the limit flow splits off a line and $f_i$ converges to the coordinate function of the $\IR$ factor. 
\end{proof}

\subsection{Existence of long necks}

\begin{Definition}[{$(\eps,m)$-center}]
   Let $(N^k,g)$ be a smooth manifold. We say that a point $z\in N^k$ is an $(\eps,m)$-center, if there is a smooth embedding 
\[
    \Phi:B((\bar x,0^{k-m}), 1/\eps)
    \to N,
\]
where $(\bar x,0^{k-m})\in \IS^m\times \IR^{k-m}$, such that 
\begin{itemize}
    \item $\Phi(\bar x,0^{k-m})=z$;
    \item For $\tilde g = R(z)g,$ we have
    \[
    \left\|\Phi^*\tilde g - \bar g\right\|_{C^{[1/\eps]}}  < \eps,
\]
where $\bar g$ denotes the standard round metric on $\IS^m\times \IR^{k-m}$ with constant scalar curvature $1.$
\end{itemize} 
\end{Definition}

\begin{Proposition}
\label{prop: long necks}
    Let $(M^4,g,f)$ be a noncollapsed steady soliton satisfying \eqref{cond: Ric and pinch}. Suppose that the tangent flow at infinity is $\IS^2\times \IR^2$, and $f$ grows linearly. Then for any $\eps>0$, $M_s:=\{f=s\}$ contains an $(\eps,2)$-center, if $s\ge \underline{s}(\eps).$ 
\end{Proposition}
    
\begin{proof}
    By the proof of Lemma 2.4 in \cite{BCMZ21}, for any large $s>0$, there is $x_s\in M_s$, such that $\ell(\Phi_{\tau_s}(x_s),\tau_s)\le C,$ where $\ell$ is Perelman's $\ell$-function based at $(o,0)$, and $cs\le \tau_s\le s/c$ for some universal constants $c,C>0.$ Here, we used the fact that $f$ grows linearly.
    
    \medskip
    \noindent\textbf{Claim 1:} For any sequence $s_i\to \infty,$ by passing to a subsequence, 
    \[
        (M^4,s_i^{-1}g, x_{s_i})
        \to (\IS^2\times \IR^2, \bar x)
    \]
    up to scaling.
    \begin{proof}[Proof of Claim 1]
        By passing to a subsequence, we may assume that $\tau_{s_i}/s_i\to a$, for some $a\in [c,1/c].$ Claim 1 follows by the existence of Perelman's asymptotic shrinkers (Theorem \ref{thm: zero AVR}) and the fact that the asymptotic shrinker (or equivalently the tangent flow at infinity) on such steady solitons is unique (\cite{BCDMZ21}).
    \end{proof}

    \medskip
    \noindent\textbf{Claim 2:} If $s\ge \underline{s}$, then
    \[
        1/C \le s\cdot R(x_s) \le C,
    \]
    for some constant $C<\infty.$
    \begin{proof}[Proof of Claim 2.]
        Suppose to the contrary that there is a sequence $s_i\to \infty,$ such that $s_iR(x_{s_i})\to \infty$ (or $0$).
        By Claim 1, $s_i R(x_{s_i})$ converges to a positive constant, by passing to a subsequence. This is a contradiction.
    \end{proof}
    \medskip
    \noindent\textbf{Claim 3:} If $s\ge \underline{s},$
    \[
        1/C \le s\cdot \bar R(x_s) \le C,
    \]
    for some constant $C<\infty,$ where $\bar R$ denotes the scalar curvature of the submanifold $M_s.$
    \def \bn {\mathbf{n}}
    \begin{proof}[Proof of Claim 3.]
    We write $\bn:=\frac{\nabla f}{|\nabla f|}.$ Recall that the second fundamental form of $M_s$ is given by
    \[
        {\rm II}(X,Y)=
        \langle \nabla_X \bn, Y \rangle=\frac{1}{|\nabla f|}\Ric(X,Y),
    \]
    for any $X,Y$ tangent to $M_s.$
        By the Gauss formula (see, e.g., \cite{DZ20a}),
        \[
            \bar R = R - \Ric(\bn,\bn)
            + H^2 - |{\rm II}|^2,
        \]
        % \textcolor{red}{
        % \[
        %     \bar R = R - 2\Ric(\bn,\bn)
        %     + H^2 - |{\rm II}|^2,
        % \]
        % }
        where $H={\rm tr}_{M_s}({\rm II})$ is the mean curvature of $M_s.$ By \cite[(4.14)]{CDM22}, for large $s,$ we have
        \[
            |{\Ric(\bn,\bn)}|
            \le CR^2
        \]
        at $x_s,$ due to the derivative estimates given by Theorem \ref{cor: curv Shi}. Thus, for large $s,$ 
        \[
            |R(x_s)-\bar R(x_s)|\le CR^2(x_s).
        \]
        Claim 3 follows by Claim 2.
    \end{proof}
\medskip
    \noindent\textbf{Claim 4:} For any $\eps>0$, $x_s$ is an $(\eps,2)$-center in $M_s$, if $s\ge \underline{s}(\eps).$
\begin{proof}[Proof of Claim 4.]
   Suppose to the contrary that there is a sequence $s_i\to \infty,$  such that $x_i:=x_{s_i}$ is not an $(\eps,2)$-center in $M_i:=M_{s_i}$, for some fixed $\eps>0.$ Combining Claim 1 and Claim 3, 
   \[
        (M^4, r_i^{-2} g, x_i)
        \to (\IS^2\times \IR^2,\bar x),
   \]
   where $r_i^{-2}=\bar R(x_i).$
   By Lemma \ref{lem: conv to a shrinker}, $(f-f(x_i))/r_i$ converges to the coordinate function of one $\IR$ factor. Thus, $M_i$, after rescaling by $r_i^{-2}$, is arbitrarily close to $\IS^2\times \IR$ locally smoothly.
   This is a contradiction.
\end{proof}
The conclusion of the proposition follows from Claim 4.
\end{proof}

\subsection{Ruling out compact reductions}
Our goal in this subsection is to rule out compact limits.
\begin{Proposition}
\label{prop: no cpt reduct}
   $(M^4,g,f)$ cannot dimension reduce to compact flows, and thus $\IS^3/\Gamma, M^{3}_{\rm Per}$ are not possible.   
\end{Proposition}

Using a similar argument as in Lemma \ref{lem: conv to a shrinker}, we can also prove the following Lemma.
\begin{Lemma}
\label{lem: reduce to cpt}
    Let $(M^4,g,f)$ be a noncollapsed steady soliton. Suppose that for some $x_i\to \infty,$
    if $r_i^{-2}=R(x_i),$
    \[
        (M,r_i^{-2}g_{r_i^2t}, x_i)
        \to (\IR\times N^3,dz^2+N^3, (0,\bar x)),
    \]
    locally smoothly over $(-\infty,0],$ where $N$ is compact. Then
    \[
        R(x_i)\to 0,
    \]
    and
    \[
        f_i := r_i^{-1}(f-f(x_i))
    \]
    converges locally smoothly to the coordinate function $z$.
\end{Lemma}
\begin{proof}
    Suppose to the contrary that $r_i^{-2}\ge a$, for some constant $a>0$, by possibly passing to a subsequence.
    By \cite[Lemma 9.1]{CDM22}, since $r_i^{-2}\ge a,$ 
    $(N^3,h_t)$ also admits a steady gradient Ricci soliton structure. But any steady soliton structure on a compact manifold is Ricci-flat, and thus $(N^3,h_t)$ is flat since $N$ is three-dimensional. This is a contradiction to the fact that $R_{h_0}(\bar x)=1.$ The rest of the proof is the same as in Lemma \ref{lem: conv to a shrinker}.
\end{proof}

\begin{Lemma}
    Suppose that $(M^4,g,f)$ dimension reduces to a compact flow $(N^3,h_t)_{t\le 0}$. Then $f$ grows linearly.
\end{Lemma}
\begin{proof}
    Let $x_i\to \infty$ such that, for $r_i^{-2}=R(x_i),$
    \[
        (M,r_i^{-2}g_{r_i^2t}, x_i)
        \to (\IR\times N^3,dz^2+h_t, (0,\bar x)),
    \]
    locally smoothly over $(-\infty,0],$ where $N^3$ is compact. 
    By Lemma \ref{lem: reduce to cpt}, $f_i=(f-f(x_i))/r_i$ converges to $z.$ There exist $y_i\in M$ such that
    \[
        (M,r_i^{-2}g_{r_i^2t}, y_i)
        \to (\IR\times N^3,dz^2+h_t, (1,\bar x)).
    \]
    In particular, $f_i(y_i)\to 1$, and thus $f(y_i) \ge f(x_i) + r_i/2,$ for large $i.$
    
    Moreover, the level sets $\{f=f(x_i)\}$, after scaling, converge to $N.$ In particular, $M_i=\{f=f(x_i)\}$ is compact, for large $i.$ 
    $M_i$ separates $M$ into two disioint parts: $\{f<f(x_i)\}$ and $\{f>f(x_i)\}.$ 
    Munteanu and \v{S}e\v{s}um showed in \cite[Corollary 5.2]{MS13} that for $r\ge 1$
    \[
        \sup_{\partial B_r(o)} f \ge 
        r - C \sqrt{r}.
    \]
    Thus, $\{f\ge f(x_i)\}$ is noncompact. 
    Munteanu and Wang proved that $M$ is connected at infinity, see \cite[Corollary 1.1]{MW11}, and thus $\{f\le f(x_i)\}$ must be compact, for large $i.$ Therefore, $f$ is bounded from below, and by Lemma \ref{lem: reduce to cpt}, $r_i\to \infty$ and $f(y_i)\to \infty.$ 
    By the same arguements as above, $\{f\le f(y_i)\}$ must be compact, for large $i.$ 
    It follows that $f$ is proper and by Proposition \ref{prop: linear growth cond}, $f$ grows linearly.
\end{proof}

We can now prove Proposition \ref{prop: no cpt reduct}. The argument is similar to \cite[Proposition 2.2]{BDS21}.
\begin{proof}[Proof of Proposition \ref{prop: no cpt reduct}]
    Let $x_i\to \infty$ such that, for $r_i^{-2}=R(x_i),$
    \[
        (M,r_i^{-2}g_{r_i^2t}, x_i)
        \longrightarrow (\IR\times N^3,dz^2+h_t, (0,\bar x)),
    \]
    locally smoothly over $(-\infty,0],$ where $N^3$ is compact. By the Lemma above, $f$ grows linearly, and $M_i:=\{f=f(x_i)\},$ after rescaling by $r_i^{-2}$, converges to $N$.
    It follows that
    \[
        \limsup_{i\to \infty}
        (\max_{M_i} R) {\rm diam}(M_i)^2 
        \le 2\, {\rm diam}(N,h_0)^2 < \infty.
    \]
    However, by Proposition \ref{prop: long necks}, $M_i$ contains arbitrarily long necks (after rescaling by the scalar curvature), if $i$ is sufficiently large. This is a contradiction.
\end{proof}

\begin{proof}[Proof of Theorem \ref{thm: possible limits 4D}]
    Theorem \ref{thm: possible limits 4D} follows by combining Propsition \ref{prop: S^2 x R / Z_2} and Propsition \ref{prop: no cpt reduct}.
\end{proof}

% We record another Lemma extending the closeness with $\IS^2\times \IR, M^3_{\rm Bry}\times\IR$ to positive time-intervals.
% \begin{Lemma}[Lai]
% Fix $o\in M.$
%     For any $\eps>0,$ there is $D_\eps<\infty$ with the following property.
%     For any $x\in M\setminus B(o,D_\eps)$ with $r^{-2}=R(x),$ either
%     \begin{itemize}
%         \item 
%         $R(x)\le \eps$, and $(M,r^{-2}g_{r^2t},(x,0))$ is $\eps$-close to $(\IS^2_{\sqrt{2}-t}\times \IR^2, (\bar x,0))$, for $t\in [-1/\eps,\sqrt{2}-\eps]$;
%         or
%         \item
%         $(M,r^{-2}g_{r^2t},(x,0))$ is $\eps$-close to $(M^3_{\rm Bry}\times\IR, g^{\rm Bry}_t+dz^2, (\bar x,0))$, for $t\in [-1/\eps,1/\eps]$.
%     \end{itemize}
% \end{Lemma}

% \begin{Lemma}
%     For any $\eps>0,$ there exist $\delta>0,D<\infty$ with the following property.
%     If $x\notin B(o,D)$ is a $(\delta,2)$-center, then each $\Phi_{t}(x)$ is an $(\eps,2)$-center, for any $t\le 0.$
% \end{Lemma}

\section{Splitting of Tangent Flows at Infinity}
\begin{Theorem}
\label{thm: split tan flow}
    Let $(M^n,g,f)$ be a complete noncollapsed non-Ricci-flat steady soliton with bounded curvature.
    Then any tangent flow at infinity splits off a line.
\end{Theorem}

\begin{proof}
The argument is basically identical to that of \cite{BCDMZ21}. Let $\left(\mathcal{X},(\nu^{\infty}_{t})_{t<0}\right)$ be a tangent flow at infinity of the canonical flow $(M^n,g_t)_{t\in(-\infty,+\infty)}$ of the steady soliton.
Fix $o\in M$, and let $\tau_i\to \infty$ such that 
\[
        \left((M^n,g^i_t)_{t\le 0},
        (\nu_{o,-\tau_it})\right)
        \xrightarrow{\quad\mathbb{F}\quad}
        \left(\mathcal{X},
        (\nu^{\infty}_{t})_{t<0}\right),
\]
where $g^i_t:=\tau_i^{-1}g_{\tau_it}$. By \cite[Theorem 2.3]{Bam20c}, there must be $z_\infty\in \mathcal{R}_\infty$ (the regular part of $\XX$), so that $z_\infty$ is a point of smooth convergence, i.e., there exist $z_i\in M$ such that
\[
    \big(M, g^i_t, (z_i,-1)\big) \longrightarrow (\mathcal{R}_\infty, g^\infty_t, z_\infty),
\]
locally smoothly. So,
    \[
        \tau_i R_{g_{-\tau_i}}(z_i) = R_{g^i_{-1}}(z_i) \to R_{g^\infty_{-1}}(z_\infty) < \infty.
    \]
    Since $|\nabla_{g_t} f_t|^2+R_{g_t}=1$, where $f_t=f\circ\Phi_t$, on $M$ by \eqref{steady-soliton-canonical-flow-identity}, we have
    \[
        \beta_i^{-2}:=
        |\mathrm{d}f|_{g^i_{-1}}^2(z_i)
        =\tau_i |\nabla_{g_{-\tau_i}} f_{-\tau_i}|^2_{g_{-\tau_i}}(z_i)
        = \tau_i\left(1 - R_{g_{-\tau_i}}(z_i)\right)
        \to \infty.
    \]
    % Let $\beta_i$ be the sequence such that $ \tau_i |\nabla_{g_{-\tau_i}} f_{-\tau_i}|^2_{g_{-\tau_i}}(z_i)=\beta_i^{-2}$, and
    Let us define $$\tilde{f}_i=\beta_i\left(f_{-\tau_i}(\cdot)-f_{-\tau_i}(z_i)\right).$$ 
    Then $\tilde{f}_i(z_i)=0$,  $|d\tilde{f}_i|_{g^i_{-1}}(z_i)=1$,   $\nabla_{g^i_{-1}}^2\tilde{f}_i=\beta_i \Ric_{{g^i_{-1}}}\to 0$. Moreover, applying the standard  derivative estimates to \eqref{steady-soliton-canonical-flow-identity}, we have that the higher derivatives $\nabla^{k+2}_{g^i_{-1}}\tilde{f}_i=\beta_i\nabla^{k}_{g^i_{-1}} \Ric_{g^i_{-1}}$ are uniformly bounded on compact sets of $\mathcal{R}_{\infty}$. Hence $\tilde{f}_i$ subconverges smoothly and locally to a function $\tilde{f}_\infty$  on $\mathcal{R}_{\infty,-1}$ (that is, the $t=-1$ time-slice of $\mathcal{R}_\infty$) satisfying $|d\tilde{f}_\infty|_{g^\infty_{-1}}(z_\infty)=1$ and $\nabla^2_{g^\infty_{-1}}\tilde{f}_\infty\equiv0$. This implies that $\mathcal{R}_{\infty,-1}$ splits off a line. Furthermore,  by \cite[Theorem 15.28(c)]{Bam20c}, $\XX_{-1}$ splits off a line. Finally, since $\XX$ is a metric soliton, $\XX$ must also split off a line.
\end{proof}

% \begin{Corollary}
%     Let $(M^5,g,f)$  be a complete noncollapsed steady soliton with nonnegative isotropic curvature (NIC). 
%     Then it has a unique tangent flow at infinity.
% \end{Corollary}

With the general splitting result, we prove the following, generalizing the previous results in \cite[Theorem 1.3]{BCDMZ21} to any dimensions.
\begin{Theorem}
\label{thm: cyl flow}
    Let $(M^n,g,f)$  be a complete noncollapsed steady soliton with bounded curvature.
    If one tangent flow at infinity is $(\IS^{n-1}/\Gamma)\times \IR$, then it is the unique one, and for any $\eps>0$, there is a compact set $K_\eps$ such that every point outside $K_\eps$ is a center of an $\eps$-neck. Moreover, outside $K_\eps,$
    \[
        \Rm>0,\qquad
        r R = \tfrac{n-1}{2} + o(1),
    \]
    as $x\to \infty,$
    where $r(x)=|ox|$ for some  point $o\in M.$
    If additionally $\sec\ge 0,\Ric>0,$ then $\Gamma=1$ and $(M^n,g,f)$ is isometric to the Bryant soliton, up to rescaling.
\end{Theorem}

The following lemma is the key ingredient of the proof of Theorem \ref{thm: cyl flow}.

\begin{Lemma}\label{lem: sphere isolated}
    Let $\Gamma$ be a discrete group acting freely on $\mathbb{S}^n$. For any $Y>0$, there is a constant $\epsilon(n,\Gamma,Y)>0$ with the following property. Let $(\XX,(\mu_t)_{t\in I})$ be a metric flow pair, where $I=[-2,-1/2]$, satisfying
    \begin{enumerate}[(1)]
        \item  $(\XX,(\mu_t)_{t\in I})$ splits off a line;
        \item $(\XX,(\mu_t)_{t\in I})$ arises as a tangent flow at infinity of some smooth ancient Ricci flow $(M^{n+1},g_t)_{t\le 0}$ with $\inf_\tau\mathcal{N}_{o,0}(\tau)>-Y$.
        \item 
        \begin{align*}
        \dist_{\mathbb{F}}^{\{-1\}}\bigg((\XX,(\mu_t)_{t\in I}), \Big((\mathbb{S}^n/\Gamma\times\mathbb{R},\bar{g}_t)_{t\in I},(\bar{\mu}_t)_{t\in I}\Big)\bigg)<\epsilon,
    \end{align*}
    where $\bar{g}_t$ is the standard shrinking metric, and $\bar{\mu}_t$ is the standard conjugate heat flow on the cylinder defined by $d\bar\mu_t=C(|\Gamma|,n)(4\pi|t|)^{-\frac{n}{2}}\cdot(4\pi|t|)^{-\frac{1}{2}}\exp\left(-\frac{|x|^2}{4|t|}\right)d\bar{g}_t$.
    \end{enumerate} 
   Then $(\XX,(\mu_t)_{t\in I})$ is isometric to $\Big((\mathbb{S}^n/\Gamma\times\mathbb{R},\bar{g}_t)_{t\in I},(\bar{\mu}_t)_{t\in I}\Big)$.
\end{Lemma}

\begin{proof}
    We argue by contradiction. Suppose there is a sequence of metric flow pairs $\{(\XX^i,(\mu_t^i)_{t\in I})\}_{i=1}^\infty$ satisfying conditions (1) and (2) (they may not arise as tangent flows of the same ancient flow). Assume that
    \begin{align}\label{the-nonsense}
        \dist_{\mathbb{F}}^{\{-1\}}\bigg((\XX^i,(\mu^i_t)_{t\in I}), \Big((\mathbb{S}^n/\Gamma\times\mathbb{R},\bar{g}_t)_{t\in I},(\bar{\mu}_t)_{t\in I}\Big)\bigg)\longrightarrow 0,
    \end{align}
    but none of them is isometric to $\Big((\mathbb{S}^n/\Gamma\times\mathbb{R},\bar{g}_t)_{t\in I},(\bar{\mu}_t)_{t\in I}\Big)$.

By \cite[Corollary 1.3]{CMZ23}, the convergence in \eqref{the-nonsense} is a local smooth one. Therefore, whenever $i$ is large enough, there is a large open set in $\XX^i$, which is almost isometric to a large open set of the cylindrical flow.  Since  $\XX^i$ also splits off a line, it follows that its other factor must be almost isometric to $\mathbb{S}^n/\Gamma$.

On the other hand, recall that each $\XX^i$ is also a metric soliton. Writing $\XX^i=\mathcal{Y}^i\times\mathbb{R}$, we know that each $\mathcal{Y}^i$  not only is a smooth shrinking soliton, but also is almost isometric to $\mathbb{S}^n/\Gamma$. By Huisken's result \cite{Hu85}, every shrinking soliton that is almost isometric to $\mathbb{S}^n/\Gamma$ must be isometric to it, namely, $\mathcal{Y}^i\cong\mathbb{S}^n/\Gamma$ whenever $i$ is large enough. This contradicts our assumption that  none of $(\XX^i,(\mu^i_t)_{t\in I})$ is isometric to $\Big((\mathbb{S}^n/\Gamma\times\mathbb{R},\bar{g}_t)_{t\in I},(\bar{\mu}_t)_{t\in I}\Big)$.

\end{proof}

Now we are ready to prove Theorem \ref{thm: cyl flow}.
\begin{proof}[Proof of Theorem \ref{thm: cyl flow}.]
    Suppose that one tangent flow at infinity is $(\IS^{n-1}/\Gamma)\times \IR$ and suppose to the contrary that there is a different tangent flow at infinity $\widetilde{\XX}.$
    By Theorem \ref{thm: split tan flow}, $\widetilde{\XX}=\XX\times \IR$ for some metric soliton $\XX.$ 
    By the proof of \cite[Proposition 3.2]{BCDMZ21}, $\XX\times\mathbb{R}$ and $(\IS^{n-1}/\Gamma)\times\mathbb{R}$ can be almost connected, which means that for any $\eps>0,$ there are metric solitons $\{\XX^i\}_{i=0}^N$ such that
    \[
        \XX^0 = \XX,\quad
        \XX^N = \IS^{n-1}/\Gamma,\quad
        {\rm dist}_{\mathbb{F}}^J(\XX^i\times\mathbb{R},\XX^{i+1}\times\mathbb{R})<\eps.
    \]
    This is a contradiction to Lemma \ref{lem: sphere isolated}.
    Therefore, there is a unique tangent flow at infinity for the steady soliton in question.

    The rest of the proof is the same as that of the four dimensional case in \cite{BCDMZ21}, which relies only on the fact that the tangent flow at infinity is unique.

    Finally, if we assume $\sec\ge 0, \Ric>0,$ then the level sets should be diffeomorphic to $\IS^{n-1}$ by \cite[Lemma 2.1]{DZ21}. Thus, $\Gamma=1$, and $(M,g,f)$ is asymptotically cylindrical in the sense of Brendle \cite{Bre14}.
    By \cite[Theorem 1.2]{Bre14}, $(M^n,g,f)$ is isometric to the Bryant soliton, up to rescaling, c.f. \cite{DZ20a,ZZ23}.
\end{proof}

Under the assumption of $\sec\ge 0,\Ric>0,$ we shall show that the only possible tangent flows at infinity are $\IS^3\times \IR$  and $\IS^2\times \IR^2.$

\begin{Corollary}
\label{cor: 4d tan flow}
    Let $(M^4,g,f)$ be a complete noncollapsed steady gradient Ricci soliton with $\sec\ge 0, \Ric>0$. 
    Then it admits a unique tangent flow at infinity which can only be
    \[
    \IS^3\times \IR \qquad \text{ or }
    \qquad \IS^2\times \IR^2.
    \]
\end{Corollary}
\begin{proof}
    Since $R+|\nabla f|^2=1$ and $\sec\ge 0,$ the curvature is bounded and thus we can apply Theorem \ref{thm: tan flow oo}.
    If the tangent flow at infinity is $(\IS^3/\Gamma)\times \IR$, by Theorem \ref{thm: cyl flow}, $\Gamma=1$ and  $(M^4,g,f)$ is in fact isometric to the Bryant soliton.

    It remains to rule out $(\IS^2\times_{\mathbb{Z}_2} \IR)\times \IR$.
    By the proof of Theorem \ref{thm: cyl flow}, 
    there is a sequence {$x_i\to\infty$}, such that $\{f=f(x_i)\}$, after rescaling, converge to $\IS^2\times_{\mathbb{Z}_2}\IR.$
    In particular, $\mathbb{RP}^2$ as the tip of $\IS^2\times_{\mathbb{Z}_2}\IR$ can be embedded into $\{f=f(x_i)\}$, for large $i.$
    By Lemma \ref{lem: sec level set}, $\{f=f(x_i)\}$ is diffeomorphic to either $\IS^3$ or $\IR^3.$ This is a contradiction because $\mathbb{RP}^2$ cannot be embedded into $\IS^3$ or $\IR^3.$
\end{proof}

\appendix
\section{A Compactness Theorem and Its Applications}
% Proof of Theorem \ref{thm: main}}

We shall prove a compactness theorem for a slightly larger class which includes, for example, the steady solitons with $\sec\ge 0.$
Since we mainly follow Perelman's arguments in \cite{Per02}, (see also, \cite{KL08}), we put this section in the appendix.

\def \AA {\mathcal{A}}

\begin{Definition}
For constants $n\in \mathbb{N},\mu_\infty<0,\Lambda\ge 1,$ we denote by
\[
    \AA(n,\mu_\infty,\Lambda)
\]
the class of ancient flows $(M^n,g_t)_{t\le 0}$ with the following properties.
\begin{enumerate}[(C1)]
\item For some $o\in M$,  (and thus for any $o\in M$,)
\[
    \inf_{\tau>0} \NN_{o,0}(\tau)\ge \mu_\infty.
\]
    \item Hamilton's trace Harnack inequality holds, i.e.,
    \begin{equation}
        \tag{Harnack}
        \label{cond: trace Harnack}
        \partial_t R - 2\nabla_X R
        + 2\Ric(X,X)\ge 0,
    \end{equation}
    on $M\times(-\infty,0],$
    for any vector field $X.$
    \item 
    \begin{equation}
\tag{Curv}
\label{cond: Ric and pinch}
    \Ric\ge 0,\quad
    |{\Rm}|\le \Lambda R,
\end{equation} 
everywhere on $M\times (-\infty,0].$
\item $g_t$ has bounded curvature over compact time intervals. (The bound may not be uniform.)

\end{enumerate}
\end{Definition}

\def \Aclass {\mathcal{A}(n,\mu_\infty,\Lambda)}

We shall prove a (smooth) pre-compactness theorem for $\AA(n,\mu_\infty, \Lambda)$ following Perelman's arguments in \cite[Section 11]{Per02}. 

\begin{Theorem}
\label{thm: cpt}
$\AA(n,\mu_\infty,\Lambda)$ is pre-compact under the Cheeger-Gromov-Hamilton convergence, after parabolic rescaling by the scalar curvature.

To be more specific,
    let $(M^n_i,g^i_t)_{t\le 0}\in \AA(n,\mu_\infty,\Lambda), r_i>0,$ and $x_i\in M_i$, such that  $R_{g^i_0}(x_i)\le r_i^{-2}.$
    Then by passing to a subsequence, 
    \[
        (M^n_i, r_i^{-2}g^i_{r_i^2t}, (x_i,0))
        \to (M^n_\infty, g^\infty_{t}, (x_\infty,0)),
    \]
    locally smoothly over $(-\infty,0].$ $(M^n_\infty, g^\infty_{t})_{t\le 0}$ is noncollapsed and satisfies \eqref{cond: trace Harnack} and \eqref{cond: Ric and pinch} with the same constants $\mu_\infty,\Lambda$, but $g^\infty_{t}$ may have unbounded curvature.
\end{Theorem}

\begin{Remark}
\label{rmk: steady Harnack}
    As observed in \cite[Theorem 1.13]{MZ21}, \eqref{cond: trace Harnack} holds on a steady soliton if $\Ric\ge 0.$
Thus, the canonical Ricci flow induced by any noncollapsed steady soliton with $\sec\ge 0$ is in $\AA(n,\mu_\infty, n)$, for some $\mu_\infty<0.$
\end{Remark}

\begin{Remark}
    By \cite[Theorem 1.13]{MZ21}, assuming the other conditions, condition (C1) is equivalent to Perelman's notion of $\kappa$-noncollapsing. This was claimed by Perelman under (not essentially) stronger curvature conditions.
    Recall that an ancient flow $(M^n,g_t)_{t\le 0}$ is (strongly) $\kappa$-noncollapsed, for some $\kappa>0,$ if for any $(x,t)\in M\times (-\infty,0]$, 
    \[
        \sup_{B_t(x,r)} R \le r^{-2}
        \implies |B_t(x,r)|_{g_t}\ge \kappa\, r^n.
    \]
\end{Remark}

Our approach to prove Theorem \ref{thm: cpt} is similar to Ni's arguments for ancient K\"ahler-Ricci flows.  Similar arguments also apply to steady solitons. See, e.g., \cite{DZ18,CDM22}.

First, we prove that any ancient flow in this class has zero aymptotic volume growth (AVR), which relies on the fact that any non-flat shrinker has zero AVR. This was proved by Carrillo-Ni in \cite[Corollary 1.1]{CaN10}. Recall that the AVR of a Riemannian manifold $(M^n,g)$ with $\Ric\ge 0$ is defined as
\[
    {\rm AVR}(g) := \lim_{r\to\infty}
    \frac{|B(o,r)|}{r^n},
\]
where $o\in M$. Clearly, AVR does not depend on the choice of $o.$

We need the following simple Lemma.
\begin{Lemma}
\label{lem: AVR decreasing}
    Let $(M^n,g_t)_{t\in[-T,0]}$ be a complete Ricci flow with
    \[
        0\le \Ric\le \Lambda,
    \]
    uniformly on $M\times [-T,0]$, for some $\Lambda\ge 1,T>0.$ Then  
    \[
        {\rm AVR}(g_{-T})\ge {\rm AVR}(g_0).
    \]
\end{Lemma} 
\begin{proof}
    By Perelman's distance expanding estimates, \cite[Lemma 8.3]{Per02},
    \[
        \partial_t |xy|_t
        \ge -5n\Lambda,
    \]
    for any $t\in [-T,0],x,y\in M.$ So
    \[
        |xy|_{0}-|xy|_{-T}
        \ge -5n\Lambda T,
    \]
    and thus
    \[
        B_0(x,r)\subseteq B_{-T}(x, 5n\Lambda T + r),
    \]
    for any $r>0.$
    For any measurable subset $\Omega\subseteq M,$ and any $t\in [-T,0],$
    \[
        \partial_t |\Omega|_t
        = -\int_{\Omega}R\,dg_t
        \le 0.
    \]
    Thus,
    \[
        |B_{-T}(x, 5n\Lambda T + r)|_{-T}
        \ge |B_0(x,r)|_0.
    \]
    The conclusion follows by dividing $r^n$ and then taking $r\to \infty.$
\end{proof}

\begin{Theorem}
\label{thm: zero AVR}
    Let $(M^n,g_t)_{t\le 0}\in \AA(n,\mu_\infty,\Lambda).$ 
    % Suppose that the condition \eqref{cond: Ric and pinch} holds everywhere on $M\times (-\infty,0]$.
    % Suppose further that the trace Harnack inequality holds, i.e.,
    % \begin{equation}
    % \tag{Harnack}
    % \label{cond: trace Harnack}
    %     \partial_t R - \nabla_X R + \Ric(X,X)\ge 0.
    % \end{equation}
    Then Perelman's asymptotic shrinkers  exist, and
    % If
    % \[
    %     {\rm AVR}(g_t)\ge {\rm AVR}(g_0),
    % \]
    % for each $t\le 0,$
    % then
    \[
        {\rm AVR}(g_0)=0.
        % \quad {\rm ASCR}(g_0)=\infty.
    \]  
\end{Theorem}
\begin{proof}
    The proof is the same as that of \cite[Corollary 5.2]{MZ21}. We rely on the fact that non-flat shrinkers have zero AVR, which was proved by Carrillo and Ni in \cite[Corollary 1.1]{CaN10}.
    We briefly sketch the proof.

     By Perelman's arguments, (see, e.g., \cite{MZ21},) Perelman's asymptotic shrinkers exist assuming the Harnack inequality and \eqref{cond: Ric and pinch}.
     That is, for any $\tau_i\to \infty,$ taking $p_i$ such that $\ell(p_i,\tau_i)\le n/2,$ 
     \[
        (M^n, \tau_i^{-1}g_{\tau_i t}, (p_i,-1))
        \to (M_\infty^n,g^\infty_t,(p_\infty,-1)),
     \]
     for some shrinker $(M_\infty^n,g^\infty_t),$ over $(-\infty,0).$ Here, $\ell$ refers to Perelman's $\ell$-function in \cite[Section 11]{Per02}.
     Suppose to the contrary that ${\rm AVR}(g_0)\ge c,$ for some $c>0$. 
     By Lemma \ref{lem: AVR decreasing}, for any $t\le 0,$
     \[
        {\rm AVR}(g_t)\ge {\rm AVR}(g_0)\ge c.
     \]
     Passing to the limit, ${\rm AVR}(g^\infty_{-1})>0,$ which is a contradiction to \cite[Corollary 1.1]{CaN10}.
     
\end{proof}

For completeness, we carry out Perelman's arguments based on the result above about zero AVR.
For simplicity, we write
\[
    |B(x,r)|_g := |B_g(x,r)|_g.
\]

\begin{Proposition}
\label{prop: reg rad by vol rad}
Let $(M^n, g_t)_{t\le 0}\in \AA(n,\mu_\infty,\Lambda).$
    % Let $(M^n,g,f)$ be a complete steady soliton satisfying condition \eqref{cond: Ric and pinch}. 
    Then for any $\theta\in (0,1),$ there is a constant $C=C(n,\mu_\infty,\Lambda,\theta)$ with the following property.
    If $x\in M,s>0$ satisfy
    \[
        \frac{|B(x,s)|_{g_0}}{s^n} \ge \theta,
    \]
    then
    \[
        s^2 R \le C,\quad \text{ on } B_{0}(x,s).
    \]
    As a consequence of \eqref{cond: trace Harnack}, 
    \[
         s^2 R \le C,\quad \text{ on } B_{0}(x,s)\times (-\infty,0].
    \]
\end{Proposition}

\begin{proof}
    Suppose not. Then there exist $(M_i^n,g^i_t)_{t\le 0}\in \AA(n,\mu_\infty,\Lambda), \bar x_i, y_i\in M_i, \bar s_i>0,$ 
    for some fixed $n,\mu_\infty,\Lambda,$
    such that 
    \[
        \frac{|B(\bar x_i,\bar s_i)|_{g^i_0}}{\bar s_i^n} \ge \theta,\quad
        |\bar x_i y_i|_{g^i_0}\le  s_i,
        \quad
        \bar s_i^2 R_{g^i_0}(y_i)\to \infty.
    \]
    In the following, we omit the subindecies $g^i_0$ when it is clear from the context.
    By a standard point picking, e.g., \cite[Lemma H.1]{KL08}, there exist $s_i\le \bar s_i, x_i\in B(\bar x_i,8\bar s_i)$ such that
    \[
        s_i^2R(x_i) \ge \bar s_i^2 R(y_i)\to \infty,
    \]
    and
    \[
        R\le 2R(x_i),\quad
        \text{ on } B(x_i,s_i).
    \]
    Moreover, since $|x_i\bar x_i|<8\bar s_i,s_i\le \bar s_i$, by the volume monotonicity, 
    \begin{equation}
    \label{eq: contr assump vol ratio}
        \frac{|B(x_i,s_i)|}{s_i^n}
        \ge \frac{|B(x_i,10\bar s_i)|}{(10\bar s_i)^n}
        \ge \frac{|B(\bar x_i,\bar s_i)|}{(10\bar s_i)^n}
        \ge 10^{-n}\theta>0.
    \end{equation}  
    Let
    \[
        r_i^{-2}:=R_{g^i_0}(x_i),\quad
        \tilde g^i_{t} := r_i^{-2} g_{r_i^2t}^i.
    \]
    By \eqref{cond: trace Harnack},
    \[
        R_{\tilde g^i} \le 2,\quad
        \text{ on } B_{\tilde g^i_0}(x_i,s_i/r_i)\times (-\infty,0].
    \]
    Recall that $s_i/r_i\to \infty.$
    Since each ancient flow $\tilde g^i$ is noncollapsed with a uniform entropy bound, by Hamilton's compactness theorem, (e.g. \cite[Corollary E.2]{KL08},)
    \[
        (M_i^n, \tilde g^i_{t},x_i)\to (M_\infty^n, g^\infty_{t}, x_\infty),
    \]
    locally smoothly over  $(-\infty,0], $  by possibly passing to a subsequence.
    
    Clearly, $(M_\infty^n, g^\infty_{t})\in \AA(n,\mu_\infty,\Lambda).$
    % Since $g_t=\Phi_t^*g,$ where $\Phi_t$ is the $1$-parameter group of diffeomorphisms generated by $-\nabla f,$ we have
    % \begin{equation}
    % \label{eq: vol at diff times}
    %     |B_t(x,r)|_t = |B(\Phi_t(x),r)|,
    % \end{equation} 
    % for any $x\in M,r>0,t\in\IR.$
    By \eqref{eq: contr assump vol ratio} and the volume monotonicity,
    \[
        \frac{|B(x_i,A)|_{\tilde g^i_0}}{A^n}
        \ge 10^{-n}\theta,
    \]
    for any $A<\infty,$ if $i$ is large. Thus, ${\rm AVR}(g^\infty_0)>0.$
   This is a contradiction to Theorem \ref{thm: zero AVR}.
    
\end{proof}

We then prove a compactness result with additional assumptions on volume noncollapsing.
\begin{Lemma}[{\cite[Corollary 42.1-(3)]{KL08}}]
\label{lem: weak cpt}
 Let $(M^n_i,g^i_t)_{t\le 0}\in \AA(n,\mu_\infty,\Lambda),$ and $x_i\in M_i.$ 
 If
 \[
    R_{g^i_0}(x_i)\le 1,\quad
    \frac{|B(x_i,r)|_{g^i_0}}{r^n}\ge \theta,
 \]
 for some $r,\theta>0,$ independent of $i,$ then by passing to a subsequence,
 \[
    (M_i^n,g^i_{t},(x_i,0))\to (M^n_\infty, g^\infty_{t}, (x_\infty,0)),
 \]
locally smoothly over $(-\infty,0]$.
\end{Lemma}
\begin{proof}
    The proof is a direct application of Proposition \ref{prop: reg rad by vol rad} and the standard volume comparison. 
\end{proof}

The weak compactness theorem implies the stronger version, Theorem \ref{thm: cpt}.

\begin{proof}[Proof of Theorem \ref{thm: cpt}]
By rescaling, we may assume that $R_{g^i_0}(x_i)=1.$
    It suffices to verify that there is a constant $\theta>0$ independent of $i$ such that
    \[
        |B(x_i,1)|_{g^i_0}\ge \theta,
    \]
    for all $i.$ Then we may apply the weaker version, Lemma \ref{lem: weak cpt} above.

    Suppose to the contrary that, by passing to a subsequence,
    \begin{equation}
    \label{eq: vol go to 0}
        |B(x_i,1)|_{g^i_0}\to 0.
    \end{equation} 
     By the volume mononicity, there is $s_i<1$ such that
    \[
        \frac{|B(x_i,s_i)|}{s_i^n}
        = \frac{\omega_n}{2},
    \]
    where $\omega_n$ denotes the volume of the standard unit ball in $\IR^n.$ 
    
    We claim that $s_i\to 0.$ 
    \begin{proof}[Proof of the claim.]
        Suppose to the contrary that $s_i\ge a $ for some $a\in (0,1).$
        Then
        \[
        \frac{\omega_n}{2}=
        \frac{|B(x_i,s_i)|}{s_i^n}
        \le  \frac{|B(x_i,a)|}{a^n}
        \le \frac{|B(x_i,1)|}{a^n} \to 0,
        \]
        which is a contradiction.
    \end{proof}
    We consider
    \[
        \tilde g^i_t := s_i^{-2}g^i_{s_i^2t}.
    \]
    For the rescaled metric,
    \[
        R_{\tilde g^i_0}(x_i) \le s_i^2\to 0,\quad
        |B(x_i,1)|_{\tilde g^i_0}
        = \tfrac{1}{2}\omega_n.
    \]
    By Lemma \ref{lem: weak cpt}, by passing to a subsequence,
    \[
        (M^i, \tilde g^i_t, (x_i,0))\to (M^n_\infty, g^\infty_{t}, (x_\infty,0)),
    \]
    locally smoothly over $(-\infty,0].$
    $R(x_\infty)=0$ on $M_\infty$, and thus $(M_\infty,g^\infty_t)$ must be flat due to the maximum principle and \eqref{cond: Ric and pinch}.
    By \eqref{eq: vol go to 0}, ${\rm AVR}(g^\infty_0)=0,$ which is a contradiction to the fact that the limit flow is  noncollapsed (by Perelman's $\kappa$-noncollapsedness or \cite[Theorem 6.1]{Bam20a}).
\end{proof}

As a standard consequence of the compactness theorem, we have 
\begin{Corollary}
\label{cor: curv Shi}
    Let $(M^n,g_t)_{t\le 0}\in \Aclass.$
    Then for each $k\in \IN,$
    \[
        |{\nabla^k \Rm}| \le C_k R^{\frac{k+2}{2}},
    \]
    everywhere on $M\times(-\infty,0]$, where $C_k$ depends on $k,n,\mu_\infty,\Lambda.$
\end{Corollary}

We also have Perelman's long range curvature estimate. Here, we follow Brendle's formulation in \cite[Theorem 4.2]{Bre22}.
\begin{Theorem}
    Let $(M^n,g_t)_{t\le 0}\in \Aclass.$
    There is an increasing function $\omega:[0,\infty)\to [0,\infty)$ depending on $n,\mu_\infty,\Lambda,$ such that
    \[
        R(y,t)\le R(x,t)\cdot \omega\Big(R(x,t)|xy|_t^2\Big),
    \]
    for any $x,y\in M,t\le 0.$
\end{Theorem}

With the help of Perelman's long range curvature estimate, we have the following general splitting theorem.
\begin{Theorem}
\label{thm: general split}
    Let $(M^n,g_t)_{t\le 0}\in \AA(n,\mu_\infty,n)$ with $\sec\ge 0$ on $M\times (-\infty,0].$ Then for any $x_i\to \infty,$ if $r_i^{-2}=R(x_i,0),$ then
    by passing to a subsequence,
    \[
        (M, r_i^{-2}g_{r_i^2t}, (x_i,0))
        \to (N^{n-1}\times \IR, \bar g_t+dz^2, ((\bar x,0),0)),
    \]
    in the sense of Cheeger-Gromov-Hamilton over $(-\infty,0],$
    for some noncollapsed ancient flow $(N,\bar g_t)$.
\end{Theorem}
For completeness, we include Brendle's proof of this theorem, \cite[Corollary 4.4]{Bre22}, using Perelman's long range curvature estimate.
\begin{proof}[Proof of Theorem \ref{thm: general split}]
    Let $x_i\to \infty$. Fix $o\in M.$
    By Perelman's long range curvature estimate,
    \[
        0<R(o,0)\le R(x_i,0)\cdot \omega\Big(R(x_i,0)|ox_i|_0^2\Big).
    \]
    We claim that 
    \[
        R(x_i,0) |ox_i|_0^2\to \infty.
    \]
    In fact, suppose to the contrary that $ R(x_i,0) |ox_i|_0^2\le A<\infty$, for some subsequence. Then
    \[
        R(x_i,0)\ge R(o,0)/\omega(A)>0,
    \]
    and thus $R(x_i,0)|ox_i|_0^2\to \infty$, which is a contradiction.

    Let 
    $$r_i^{-2}=R(x_i,0),\quad
    g^i_t:=r_i^{-2}g_{r_i^2t}.$$
    By the compactness, Theorem \ref{thm: cpt}, by passing to a subsequence,
    $(M, g^i_t,x_i)\to (\overline{M}, \bar g_t, \bar x)$, smoothly over $(-\infty,0].$
   By the standard theorem of splitting at infinity, e.g., \cite[Theorem 5.35]{MT07}, $(\overline{M}, \bar g_t)$ splits off a line. 
\end{proof}

\begin{proof}[Proof of Theorem \ref{thm: main}]

Theorem \ref{thm: main} follows by Remark \ref{rmk: steady Harnack} and Theorem \ref{thm: general split}.
    
\end{proof}

We mention an application of the curvature derivative estimates to steady solitons due to Deng and Zhu in \cite{DZ20a,DZ20b}.
\begin{Theorem}[Deng-Zhu]
    Let $(M^n,g,f)$ be a steady soliton such that its cononical Ricci flow is in $\Aclass$.
    Suppose that
    \[
        R(o)\ge a,
    \]
    for some $o\in M, a>0.$
    Then for any $\rho>0, x\in B(o,\rho),t\le 0,$
    \[
        (1-t)R(x,t)\ge c(n,\mu_\infty,\Lambda,\rho,a)>0.
    \]
\end{Theorem}
Our improvement here is that the constant $c$ is universal depending only on $n,\mu_\infty,\Lambda,\rho, a$. Since $\Phi_t(x)$ moves  away from $x$ linearly in $-t$, this is a inverse linear lower bound for the scalar curvature.
See also \cite[Corollary 2.7]{CCMZ23} for a inverse quadratic lower bound of $R$ for 4d steady soliton without any curvature positivity conditions. 
\begin{proof}
    Fix $o\in \CC.$  For any $x\in M, t\le 1,$ along $\Phi_t(x),$ by Corollary \ref{cor: curv Shi},
    \[
        \partial_t R(\Phi_t(x))
        = \Delta R + 2|{\Ric}|^2
        \le  CR^2,
    \]
    where $C$ depends on $n,\mu_\infty,\Lambda.$
    By integration, for $t\le 0,$
    \[
        \frac{1}{R(x,1)} - \frac{1}{R(x,t)}
        \ge -C (1-t).
    \]
    As a standard consequence of \eqref{cond: trace Harnack}, for any $x\in B(o,\rho),$
    \[
        R(x,1)
        \ge R(o)\cdot \exp\left(-{|xo|^2/2}\right)
        \ge a e^{-\rho^2/2}.
    \]
    Thus, for any $x\in B(o,\rho),t\le 0,$
    \[
        (1-t)R(x,t)\ge c(n,\mu_\infty,\Lambda,\rho,a).
    \]
    
\end{proof}

% \begin{Corollary}
%     Let $(M^n,g,f)$ be a steady soliton such that its cononical Ricci flow is in $\Aclass$.
%     Suppose that $\Ric>0$ and $\CC=\{|\nabla f|=0\}$ is nonempty. Then $\CC=\{o\}$, and for any $x\in M,$
%     \[
%         R(x) \ge \frac{c}{1+|ox|},
%     \]
%     where $c=c(n,\mu_\infty,\Lambda)>0.$
% \end{Corollary}
% \begin{proof}
% Since $\nabla^2f=\Ric>0$, $f$ is strictly convex and it has at most one critical point.
% By the proof of Proposition \ref{prop: linear growth Ric}, for any $x\notin B(o,1),$ there is $t_x>0$ such that $\Phi_{t_x}(x)\in \partial B(o,1).$

% \end{proof}

%\bibliography{bibliography}{}
\bibliographystyle{amsalpha}

\newcommand{\alphalchar}[1]{$^{#1}$}
\providecommand{\bysame}{\leavevmode\hbox to3em{\hrulefill}\thinspace}
\providecommand{\MR}{\relax\ifhmode\unskip\space\fi MR }
\providecommand{\MRhref}[2]{%
  \href{http://www.ams.org/mathscinet-getitem?mr=#1}{#2}
}

\noindent Mathematics Institute, Zeeman Building, University of Warwick, Coventry CV4 7AL, UK
\\ E-mail address: \verb"pak-yeung.chan@warwick.ac.uk"
\\

\noindent Department of Mathematics, Rutgers University, Piscataway, NJ 08854, USA
\\ E-mail address: \verb"zilu.ma@rutgers.edu"
\\

\noindent School of Mathematical Sciences, Shanghai Jiao Tong University, Shanghai, China, 200240
\\ E-mail address: \verb"sunzhang91@sjtu.edu.cn"

\end{document}